\documentclass[a4paper,12pt]{amsart}

\usepackage{amsmath}
\usepackage{amssymb}
\usepackage{mathrsfs}
\usepackage{enumerate}
\usepackage{ifthen}
\usepackage{graphicx}
\usepackage[T1]{fontenc} 

\nonstopmode \numberwithin{equation}{section}
\newtheorem{thm}{Theorem}[section]
\newtheorem{cor}{Corollary}[section]
\newtheorem{lem}{Lemma}[section]
\newtheorem{prop}{Proposition}[section]

\newtheorem{conj}{Conjecture}

\theoremstyle{definition}

\newtheorem{example}{Example}[section]
\newtheorem{prob}{Problem}[section]
\newtheorem{rem}{Remark}[section]

\newcounter{minutes}
\divide\time by 60
\newcounter{hours}
\multiply\time by 60
\addtocounter{minutes}{-\time}

\newcounter {own}
\def\theown {\thesection       .\arabic{own}}

\newenvironment{pf}[1][]{%
 \vskip 3mm
 \noindent
 \ifthenelse{\equal{#1}{}}%
  {{\slshape Proof. }}%
  {{\slshape #1.} }%
 }%
{\qed\bigskip}

\newcounter{alphabet}
\newcounter{tmp}

\def\be{\begin{equation}}
\def\ee{\end{equation}}

\newcommand{\bee}{\begin{enumerate}}
\newcommand{\eee}{\end{enumerate}}

\newcommand{\blem}{\begin{lem}}
\newcommand{\elem}{\end{lem}}
\newcommand{\bthm}{\begin{thm}}
\newcommand{\ethm}{\end{thm}}
\newcommand{\bcor}{\begin{cor}}
\newcommand{\ecor}{\end{cor}}
\newcommand{\beg}{\begin{examp}}
\newcommand{\eeg}{\end{examp}}
\newcommand{\begs}{\begin{examples}}
\newcommand{\eegs}{\end{examples}}
\newcommand{\bdefe}{\begin{defin}}
\newcommand{\edefe}{\end{defin}}
\newcommand{\bprob}{\begin{prob}}
\newcommand{\eprob}{\end{prob}}
\newcommand{\bei}{\begin{itemize}}
\newcommand{\eei}{\end{itemize}}

\newcommand{\bcon}{\begin{conj}}
\newcommand{\econ}{\end{conj}}
\newcommand{\bcons}{\begin{conjs}}
\newcommand{\econs}{\end{conjs}}
\newcommand{\bprop}{\begin{prop}}
\newcommand{\eprop}{\end{prop}}
\newcommand{\br}{\begin{rem}}
\newcommand{\er}{\end{rem}}
\newcommand{\brs}{\begin{rems}}
\newcommand{\ers}{\end{rems}}
\newcommand{\bo}{\begin{obser}}
\newcommand{\eo}{\end{obser}}
\newcommand{\bos}{\begin{obsers}}
\newcommand{\eos}{\end{obsers}}
\newcommand{\bpf}{\begin{pf}}
\newcommand{\epf}{\end{pf}}
\newcommand{\ba}{\begin{array}}
\newcommand{\ea}{\end{array}}
\newcommand{\beq}{\begin{eqnarray}}
\newcommand{\beqq}{\begin{eqnarray*}}
\newcommand{\eeq}{\end{eqnarray}}
\newcommand{\eeqq}{\end{eqnarray*}}

\newcommand{\ds}{\displaystyle}

\begin{document}

\title{Toeplitz determinants whose elements are the coefficients of univalent functions}

\author{Md Firoz Ali}
\address{Md Firoz Ali,
Department of Mathematics,
Indian Institute of Technology Kharagpur,
Kharagpur-721 302, West Bengal, India.}
\email{ali.firoz89@gmail.com}

\author{D. K. Thomas}
\address{D. K. Thomas,
Department of Mathematics,
Swansea University, Singleton Park,
Swansea, SA2 8PP, United Kingdom.}
\email{d.k.thomas@swansea.ac.uk}

\author{A. Vasudevarao}
\address{A. Vasudevarao,
Department of Mathematics,
Indian Institute of Technology Kharagpur,
Kharagpur-721 302, West Bengal, India.}
\email{alluvasu@maths.iitkgp.ernet.in}


\def\thefootnote{}
\footnotetext{ {\tiny File:~\jobname.tex,
printed: \number\year-\number\month-\number\day,
          \thehours.\ifnum\theminutes<10{0}\fi\theminutes }
} \makeatletter\def\thefootnote{\@arabic\c@footnote}\makeatother

\begin{abstract}
Let $\mathcal{S}$ denote the class of analytic and univalent functions in $\mathbb{D}:=\{z\in\mathbb{C}:\, |z|<1\}$ of the form $f(z)= z+\sum_{n=2}^{\infty}a_n z^n$. In this paper, we determine  sharp estimates for the Toeplitz determinants whose elements are the Taylor coefficients of functions in $\mathcal{S}$ and its certain subclasses. We also discuss similar problems for typically real functions.
\end{abstract}

\thanks{}

\maketitle
\pagestyle{myheadings}
\markboth{Md Firoz Ali, D. K. Thomas and A. Vasudevarao}{Toeplitz determinant}

\section{Introduction and Preliminaries}

Let $\mathcal{H}$ denote the space of analytic functions in the unit disk $\mathbb{D}:=\{z\in\mathbb{C}:\, |z|<1\}$ and $\mathcal{A}$ denote the class of functions $f$ in $\mathcal{H}$ with Taylor series
\begin{equation}\label{p-001}
f(z)= z+\sum_{n=2}^{\infty}a_n z^n.
\end{equation}
The subclass $\mathcal{S}$ of $\mathcal{A}$, consisting of univalent (i.e., one-to-one) functions  has attracted much interest for over a century, and is a central area of research in Complex Analysis. A function $f\in\mathcal{A}$ is called starlike if $f(\mathbb{D})$ is starlike with respect to the origin i.e., $tf(z)\in f(\mathbb{D})$ for every $0\le t\le 1$. Let $\mathcal{S}^*$ denote the class of starlike functions in $\mathcal{S}$. It is well-known that a function $f\in\mathcal{A}$ is starlike if, and only if,
\begin{equation*}\label{p-010}
{\rm Re\,}\left(\frac{zf'(z)}{f(z)}\right)>0, \quad z\in\mathbb{D}.
\end{equation*}
An important member of the class $\mathcal{S}^*$ as well as of the class $\mathcal{S}$ is the Koebe function $k$ defined by $k(z)=z/(1-z)^2$. This function plays the role of extremal function in most of the problems for the classes $\mathcal{S}^*$ and $\mathcal{S}$.
\bigskip

A function $f\in\mathcal{A}$ is called convex if $f(\mathbb{D})$ is a convex domain. Let $\mathcal{C}$ denote the class of convex functions in $\mathcal{S}$. It is well-known that a function $f\in\mathcal{A}$ is in $\mathcal{C}$ if, and only if,
\begin{equation*}\label{p-015}
{\rm Re\,}\left(1+\frac{zf''(z)}{f'(z)}\right)>0, \quad z\in\mathbb{D}.
\end{equation*}
From the above it is easy to see that $f\in\mathcal{C}$ if, and only if, $zf'\in\mathcal{S}^*$.
\bigskip

A function $f\in\mathcal{A}$ is said to be close-to-convex if there exists a starlike function $g\in\mathcal{S}^*$ and a real number $\alpha\in(-\pi/2,\pi/2)$, such that
\begin{equation}\label{p-020}
{\rm Re\,} \left(e^{i\alpha}\frac{zf'(z)}{g(z)}\right)>0, \quad z\in\mathbb{D}.
\end{equation}
Let $\mathcal{K}$ denote the class of all close-to-convex functions. It is well-known that every close-to-convex function is univalent in $\mathbb{D}$ (see \cite{Duren-book}). Geometrically, $f\in\mathcal{K}$ means that the complement of the image-domain $f(\mathbb{D})$ is the union of rays that are disjoint (except that the origin of one ray may lie on another one of the rays).
\bigskip

Let $\mathcal{R}$ denote class of functions $f$ in $\mathcal{A}$ satisfying ${\rm Re\,} f'(z)>0$ in $\mathbb{D}$. It is well-known that functions in $\mathcal{R}$ are close-to-convex, and hence univalent. Functions in $\mathcal{R}$ are sometimes called functions of bounded boundary rotation.
\bigskip

A function $f$  satisfiying the condition $({\rm Im\,}z) ({\rm Im\,}f(z))\ge0$ for $z\in\mathbb{D}$ is called a typically real. Let $\mathcal{T}$ denote the class of all typically real functions. Robertson \cite{Robertson-1935} proved that $f\in\mathcal{T}$ if, and only if, there exists a probability measure $\mu$ on $[-1, 1]$ such that
\begin{equation*}\label{p-200}
f(z)=\int_{-1}^{1} k(z,t)\,d\mu(t),
\end{equation*}
where
\begin{equation*}\label{p-205}
k(z,t)=\frac{z}{1-2tz+z^2}, \quad z\in\mathbb{D},\quad t\in[-1,1].
\end{equation*}
\bigskip

Hankel matrices and determinants play an important role in several branches of mathematics, and have many applications \cite{Ye-Lim-2016}. The Toeplitz determinants are closely related to Hankel determinants. Hankel matrices have constant entries along the reverse diagonal, whereas Toeplitz matrices have constant entries along the diagonal. For a good summary of the applications of Toeplitz matrices to the  wide range of areas of pure and applied mathematics, we refer to \cite{Ye-Lim-2016}. Recently, Thomas and Halim \cite{Thomas-Halim-2017} introduced the concept of the symmetric Toeplitz determinant for  analytic functions $f$ of the form (\ref{p-001}), and defined the symmetric Toeplitz determinant $T_q(n)$ as follows
$$
T_q(n):=
\begin{vmatrix}
a_n & a_{n+1} & \cdots & a_{n+q-1}\\
a_{n+1} & a_n & \cdots & a_{n+q-2}\\
\vdots & \vdots & \vdots & \vdots &\\
a_{n+q-1} & a_{n+q-2} & \cdots & a_{n}
\end{vmatrix}
$$
where $n,q= 1,2,3\ldots$ with $a_1=1$. In particular,
$$
T_2(2)=
\begin{vmatrix}
a_2 & a_3\\
a_3 & a_2
\end{vmatrix},
\quad
T_2(3)=
\begin{vmatrix}
a_3 & a_4\\
a_4 & a_3
\end{vmatrix},
\quad
T_3(1)=
\begin{vmatrix}
1 & a_2 & a_3\\
a_2 & 1 & a_2\\
a_3 & a_2 & 1
\end{vmatrix}\\,
\quad
T_3(2)=
\begin{vmatrix}
a_2 & a_3 & a_4\\
a_3 & a_2 & a_3\\
a_4 & a_3 & a_2
\end{vmatrix}.
$$

For small  values of $n$ and $q$,  estimates of the Toeplitz determinant $|T_q(n)|$ for functions in $\mathcal{S}^*$ and $\mathcal{K}$ have been studied in \cite{Thomas-Halim-2017}. Similarly,  estimates of the Toeplitz determinant $|T_q(n)|$ for functions in $\mathcal{R}$ have been studied in \cite{Radhika-Sivasubramanian-Murugusundaramoorthy-Jahangiri-2016}, when $n$ and $q$ are small.
Apart from \cite{Radhika-Sivasubramanian-Murugusundaramoorthy-Jahangiri-2016} and \cite{Thomas-Halim-2017}, there appears to be little in the literature  concerning  estimates of   Toeplitz determinants. In both
\cite{Radhika-Sivasubramanian-Murugusundaramoorthy-Jahangiri-2016, Thomas-Halim-2017}  we observe an invalid assumption  in the proofs. It is the purpose of this paper to give  estimates  for Toeplitz determinants $T_q(n)$ for functions in $\mathcal{S}$, $\mathcal{S}^*$, $\mathcal{C}$, $\mathcal{K}$,  $\mathcal{R}$, and $\mathcal{T}$, when $n$ and $q$ are small.

\bigskip

Let $\mathcal{P}$ denote the class of analytic functions $p$ in $\mathbb{D}$ of the form
\begin{equation}\label{p-030}
p(z)= 1+\sum_{n=1}^{\infty}c_n z^n
\end{equation}
such that  ${\rm\, Re\,} p(z)>0$ in $\mathbb{D}$. Functions in $\mathcal{P}$ are sometimes called Carath\'{e}odory functions. To prove our main results, we need some preliminary results for functions in $\mathcal{P}$.

\begin{lem}\cite[p. 41]{Duren-book}\label{p-lemma001}
For a function $p\in\mathcal{P}$ of the form (\ref{p-030}), the sharp inequality $|c_n|\le 2$ holds for each $n\ge 1$. Equality holds for the function $p(z)=(1+z)/(1-z)$.
\end{lem}


\begin{lem}\cite[Theorem 1]{Efraimidis-2016}\label{p-lemma010}
Let $p\in\mathcal{P}$ be of the form (\ref{p-030}) and $\mu\in\mathbb{C}$. Then
$$
|c_n-\mu c_kc_{n-k}|\le 2\max\{1,|2\mu-1|\}, \quad 1\le k\le n-1.
$$
If $|2\mu-1|\ge1$ then the inequality is sharp for the function $p(z)=(1+z)/(1-z)$ or its rotations. If $|2\mu-1|<1$ then the inequality is sharp for the function $p(z)=(1+z^n)/(1-z^n)$ or its rotations.
\end{lem}

\section{Main Results}

\begin{thm}\label{theorem-001}
Let $f\in\mathcal{S}$ be of the form (\ref{p-001}). Then
\begin{enumerate}[(i)]
\item $\displaystyle |T_2(n)|=|a_n^2-a_{n+1}^2|\le 2n^2+2n+1$ for $n\ge2$,\\[-2mm]

\item $\displaystyle |T_3(1)|\le 24$.
\end{enumerate}
Both  inequalities  are sharp.
\end{thm}

\begin{proof}
Let $f\in\mathcal{S}$ be of the form (\ref{p-001}). Then clearly
\begin{equation}\label{p-050}
|T_2(n)|=|a_n^2-a_{n+1}^2|\le |a_n^2|+|a_{n+1}^2|\le n^2+(n+1)^2=2n^2+2n+1.
\end{equation}
Equality holds in (\ref{p-050}) for the function $f$ defined by
\begin{equation}\label{p-070}
f(z):=\frac{z}{(1-iz)^2}=z+2iz^2-3z^3-4iz^4+5z^5+\cdots.
\end{equation}

Again, if $f\in\mathcal{S}$ is of the form (\ref{p-001}) then by the Fekete-Szeg\"o inequality for functions in $\mathcal{S}$, we have
\begin{align}\label{p-080}
|T_3(1)| &=|1-2a_2^2+2a_2^2a_3-a_3^2|\\
&\le 1+2|a_2^2|+|a_3||a_3-2a_2^2|\nonumber\\
&\le 1+8+(3)(5)\nonumber\\
&=24.\nonumber
\end{align}
Equality holds in (\ref{p-080}) for the function $f$ defined by (\ref{p-070}).
\end{proof}

\begin{rem}
Since the function $f$ defined by (\ref{p-070}) belongs to $\mathcal{S}^*$, and $\mathcal{S}^*\subset\mathcal{K}\subset\mathcal{S}$, the sharp inequalities in Theorem \ref{theorem-001} also hold for functions in $\mathcal{S}^*$ and $\mathcal{K}$. In particular, the sharp inequalities $|T_2(2)|\le 13$ and $|T_2(3)|\le 25$ hold for functions in $\mathcal{S}^*$, $\mathcal{K}$ and $\mathcal{S}$.
\end{rem}

\begin{thm}\label{theorem-005}
Let $f\in\mathcal{S}^*$ be of the form (\ref{p-001}). Then $|T_3(2)|\le 84$.

 \noindent The inequality is sharp.
\end{thm}

\begin{proof}
Let $f\in\mathcal{S}^*$ be of the form (\ref{p-001}). Then there exists a function $p\in\mathcal{P}$ of the form (\ref{p-030}) such that $zf'(z)=f(z)p(z)$. Equating  coefficients, we obtain
\begin{equation}\label{p-040}
a_2=c_1,\quad a_3=\frac{1}{2}(c_2+c_1^2)\quad\mbox{ and } \quad  a_4=\frac{1}{6}c_1^3+\frac{1}{2}c_1c_2+\frac{1}{3}c_3.
\end{equation}

By a simple computation $T_3(2)$ can be written as $ T_3(2)=(a_2-a_4)(a_2^2-2a_3^2+a_2a_4)$. If $f\in\mathcal{S}^*$ then clearly, $|a_2-a_4|\le |a_2|+|a_4|\le 6$. Thus we need to maximize $|a_2^2-2a_3^2+a_2a_4|$ for functions in $\mathcal{S}^*$, and so writing $a_2, a_3$ and $a_4$ in terms of $c_1, c_2$ and $c_3$ with the help of (\ref{p-040}), we obtain
\begin{align*}
|a_2^2-2a_3^2+a_2a_4| &= \left|c_1^2-\frac{1}{3}c_1^4-\frac{1}{2}c_1^2c_2-\frac{1}{2}c_2^2+\frac{1}{3}c_1c_3\right|\\
&\le |c_1|^2+\frac{1}{3}|c_1|^4+\frac{1}{2}|c_2|^2+\frac{1}{3}|c_1|\left|c_3-\frac{3}{2}c_1c_2\right|.
\end{align*}
From Lemma \ref{p-lemma001} and Lemma \ref{p-lemma010}, it easily follows that
\begin{equation}
|a_2^2-2a_3^2+a_2a_4| \le 4+\frac{16}{3}+\frac{4}{2}+\frac{2}{3}(4)=14.
\end{equation}
Therefore, $|T_3(2)|\le 84$, and the inequality is sharp for the function $f$ defined by (\ref{p-070}).
\end{proof}

\begin{rem}
In \cite{Thomas-Halim-2017}, it was claimed that $|T_2(2)|\le 5$, $|T_2(3)|\le 7$, $|T_3(1)|\le 8$ and $|T_3(2)|\le 12$ hold for functions in $\mathcal{S}^*$, and these estimates are sharp. Similar results were  also obtained for certain close-to-convex functions. For the function $f$ defined by (\ref{p-070}), a simple computation gives $|T_2(2)|= 13$ and $|T_2(3)|= 25$, $|T_3(1)|= 24$ and $|T_3(2)|=84$ which shows that these estimates are not correct. In proving these estimates the authors assumed that $c_1>0$ which is not justified, since the functional $|T_q(n)|$ $(n\ge1, q\ge2)$ is not rotationally invariant.
\end{rem}

To prove our next result we need the following results for functions in $\mathcal{S}^*$.

\begin{lem}\cite[Theorem 3.1]{Janteng-Halim-Darus-2007}\label{p-lemma020}
Let $g\in\mathcal{S}^*$ and be of the form $g(z)= z+\sum_{n=2}^{\infty}b_n z^n$. Then $|b_2b_4-b_3^2|\le 1$, and the inequality is sharp for the Koebe function $k(z)=z/(1-z)^2$, or its rotations.
\end{lem}

\begin{lem}\cite[Lemma 3]{Koepf-1987}\label{p-lemma022}
Let $g\in\mathcal{S}^*$ be of the form $g(z)=z+\sum_{n=2}^{\infty}b_n z^n$. Then for any $\lambda\in\mathbb{C}$,
$$
|b_3-\lambda b_2^2|\le \max\{1,|3-4\lambda|\}.
$$
The inequality is sharp for $k(z)=z/(1-z)^2$, or its rotations if $|3-4\lambda|\ge 1$, and for $(k(z^2))^{1/2}$, or its rotations if $|3-4\lambda|<1$.
\end{lem}

\begin{lem}\cite[Theorem 2.2]{Ma-1999}\label{p-lemma025}
Let $g\in\mathcal{S}^*$ be of the form $g(z)= z+\sum_{n=2}^{\infty}b_n z^n$. Then
$$
|\lambda b_nb_m-b_{n+m-1}|\le \lambda nm-(n+m-1) \quad\mbox{ for } \lambda\ge\frac{2(n+m-1)}{nm},
$$
where $n,m=2,3,\ldots$. The inequality is sharp for the Koebe function $k(z)=z/(1-z)^2$, or its rotations.
\end{lem}

\begin{lem}\label{p-lemma030}
Let $f\in\mathcal{K}$ be of the form (\ref{p-001}). Then $|a_2a_4-2a_3^2|\le 21/2$.
\end{lem}

\begin{proof}
Let $f\in\mathcal{K}$ be of the form (\ref{p-001}). Then there exists a starlike function $g$ of the form $g(z)= z+\sum_{n=2}^{\infty}b_n z^n$, and a real number $\alpha\in(-\pi/2,\pi/2)$, such that (\ref{p-020}) holds. This implies there exists a Carath\'{e}odory function $p\in\mathcal{P}$ of the form (\ref{p-030}) such that
$$
e^{i\alpha}\frac{zf'(z)}{g'(z)}= p(z)\cos\alpha+i\sin\alpha.
$$
Comparing  coefficients we obtain
\begin{align*}
2a_2&=b_2+c_1e^{-i\alpha}\cos\alpha\\
3a_3&=b_3+b_2c_1e^{-i\alpha}\cos\alpha+c_2e^{-i\alpha}\cos\alpha\\
4a_4&=b_4+b_3c_1e^{-i\alpha}\cos\alpha+b_2c_2e^{-i\alpha}\cos\alpha+c_3e^{-i\alpha}\cos\alpha,
\end{align*}
and a simple computation gives
\begin{align*}
72(a_2a_4-2a_3^2)&= (9b_2b_4-16b_3^2)+(9b_4-23b_2b_3)c_1 e^{-i\alpha}\cos\alpha\\
&\quad +(9b_3-16b_2^2)c_1^2e^{-2i\alpha}\cos^2\alpha +(9b_2^2-32b_3)c_2e^{-i\alpha}\cos\alpha\\
&\quad +(9c_3-23c_1 c_2e^{-i\alpha}\cos\alpha)b_2e^{-i\alpha}\cos\alpha +(9c_1c_3-16 c_2^2)e^{-2i\alpha}\cos^2\alpha.
\end{align*}
Consequently using the triangle inequality, we obtain
\begin{align}\label{p-135}
72|a_2a_4-2a_3^2|&\le |9b_2b_4-16b_3^2|+|9b_4-23b_2b_3||c_1| +|9b_3-16b_2^2||c_1^2|\\
&\quad +|9b_2^2-32b_3||c_2| +|9c_3-23c_1 c_2e^{-i\alpha}\cos\alpha||b_2| +|9c_1c_3-16 c_2^2|.\nonumber
\end{align}
By Lemma \ref{p-lemma020}, Lemma \ref{p-lemma022} and Lemma \ref{p-lemma025}, it easily follows that
\begin{align}
|9b_2b_4-16b_3^2| &\le 9|b_2b_4-b_3^2|+7|b_3|^2\le 9+63=72,\label{p-140}\\[2mm]
|9b_4-23b_2b_3| &= 9\left|b_4-\frac{23}{9}b_2b_3\right|\le 9\left(\frac{46}{3}-4\right)=102,\label{p-145}\\
|9b_3-16b_2^2| &= 9\left|b_3-\frac{16}{9}b_2^2\right|\le 9\left(\frac{64}{9}-3\right)=37,\label{p-150}\\
|9b_2^2-32b_3| &= 32\left|b_3-\frac{9}{32}b_2^2\right|\le 32\left(3-\frac{9}{8}\right)=60.\label{p-155}
\end{align}
Again, by Lemma \ref{p-lemma010}, it easily follows that
$$
|9c_3-23c_1 c_2e^{-i\alpha}\cos\alpha|= 9|c_3-\mu c_1 c_2|\le 18\max\{1,|2\mu-1|\}
$$
where $\mu=\dfrac{23}{9}e^{-i\alpha}\cos\alpha$. Now note that
\begin{align*}
|2\mu-1|^2
&= \left(\frac{23}{9}\cos2\alpha+\frac{14}{9}\right)^2 + \left(\frac{23}{9}\sin2\alpha\right)^2\\
&= \left(\frac{23}{9}\right)^2+ \left(\frac{14}{9}\right)^2 +2\left(\frac{23}{9}\right)\left(\frac{14}{9}\right) \cos2\alpha,\\
\end{align*}
and so
$$
1\le |2\mu-1|\le \frac{37}{9}.
$$
Therefore
\begin{equation}\label{p-160}
|9c_3-23c_1 c_2e^{-i\alpha}\cos\alpha| \le 74.
\end{equation}
Again by Lemma \ref{p-lemma010}, it easily follows that
\begin{equation}\label{p-165}
|9c_1c_3-16 c_2^2|\le  9|c_1c_3- c_4|+9\left|c_4-\frac{16}{9}c_2^2\right|\le 18+46=64.
\end{equation}
By Lemma \ref{p-lemma001}, and using the inequalities (\ref{p-140}), (\ref{p-145}), (\ref{p-150}), (\ref{p-155}), (\ref{p-160}) and (\ref{p-165}) in (\ref{p-135}), we obtain
$$
|a_2a_4-2a_3^2|\le \frac{1}{72}(72+204 +148 +120 +148 +64)=\frac{21}{2}.
$$
\end{proof}

\begin{thm}\label{theorem-020}
Let $f\in\mathcal{K}$ be of the form (\ref{p-001}). Then $|T_3(2)|\le 86$.
\end{thm}

\begin{proof}
Let $f\in\mathcal{K}$ be of the form (\ref{p-001}). Then by Lemma \ref{p-lemma030} we have
\begin{align*}
|T_3(2)| &=|a_2^3-2a_2a_3^2-a_2a_4^2+2a_3^2a_4|\\
&\le |a_2|^3 + 2|a_2||a_3^2| +|a_4||a_2a_4-2a_3^2|\\
&\le 8+36+42=86.
\end{align*}
\end{proof}

\begin{rem}
In Theorem \ref{theorem-005}, we have proved that $|T_3(2)|\le 84$ for functions in $\mathcal{S}^*$, and the inequality is sharp for the function $f$ defined by (\ref{p-070}). Therefore it is natural to conjecture that $|T_3(2)|\le 84$ holds for functions in $\mathcal{K}$ and that equality holds for the function $f$ defined by (\ref{p-070}).
\end{rem}

\begin{thm}\label{theorem-015}
Let $f\in\mathcal{C}$ be of the form (\ref{p-001}). Then
\begin{enumerate}[(i)]
\item $\displaystyle |T_2(n)|\le 2$ for $n\ge2$.\\[-2mm]

\item $\displaystyle |T_3(1)|\le 4$.\\[-2mm]

\item $\displaystyle |T_3(2)|\le 4$.
\end{enumerate}
All the inequalities are sharp.
\end{thm}

\begin{proof}
Let $f\in\mathcal{C}$ be of the form (\ref{p-001}). Then there exists a function $p\in\mathcal{P}$ of the form (\ref{p-030}) such that $f'(z)+zf''(z)=f'(z)p(z)$. Equating  coefficients, we obtain
\begin{equation}\label{p-115}
2a_2=c_1,\quad 3a_3=\frac{1}{2}(c_2+c_1^2)\quad\mbox{ and }\quad 4a_4=\frac{1}{6}c_1^3+\frac{1}{2}c_1c_2+\frac{1}{3}c_3.
\end{equation}
Clearly
\begin{equation}\label{p-120}
|T_2(n)|=|a_n^2-a_{n+1}^2|\le |a_n^2|+|a_{n+1}^2|\le 1+1=2.
\end{equation}
Equality holds in (\ref{p-120}) for the function $f$ defined by
\begin{equation}\label{p-125}
f(z):=\frac{z}{1-iz}=z+iz^2-z^3-iz^4+z^5+\cdots.
\end{equation}

Again if $f\in\mathcal{C}$ is of the form (\ref{p-001}) then from Lemma \ref{p-lemma010} and (\ref{p-115}), we obtain
\begin{align}\label{p-130}
|T_3(1)| &=|1-2a_2^2+2a_2^2a_3-a_3^2|\\
&\le 1+2|a_2^2|+|a_3||a_3-2a_2^2|\nonumber\\
&\le 1+2+\frac{1}{6}|c_2-2c_1^2|\nonumber\\
&\le 4.\nonumber
\end{align}
It is easy to see that  equality holds in (\ref{p-130})  for the function $f$ defined by (\ref{p-125}).
\bigskip

Next note that $ T_3(2)=(a_2-a_4)(a_2^2-2a_3^2+a_2a_4)$. If $f\in\mathcal{C}$ then clearly $|a_2-a_4|\le |a_2|+|a_4|\le 2$. Thus we need to maximize $|a_2^2-2a_3^2+a_2a_4|$ for functions in $\mathcal{C}$.

\bigskip

Writing $a_2, a_3$ and $a_4$ in terms of $c_1, c_2$ and $c_3$ with the help of (\ref{p-115}), we obtain
\begin{align*}
|a_2^2-2a_3^2+a_2a_4| &= \frac{1}{144}\left|5c_1^4-36c_1^2+7c_1^2c_2+8c_2^2-6c_1c_3\right|\\
&\le \frac{1}{144}\left(5|c_1|^4+36|c_1|^2+8|c_2|^2+6|c_1||c_3-\frac{7}{6}c_1c_2|\right).
\end{align*}
From Lemma \ref{p-lemma001} and Lemma \ref{p-lemma010}, it easily follows that
\begin{equation}
|a_2^2-2a_3^2+a_2a_4| \le \frac{1}{144}\left(80+144+32+32\right)=2.
\end{equation}
Therefore, $|T_3(2)|\le 4$, and the inequality is sharp for the function $f$ defined by (\ref{p-125}).
\end{proof}

\begin{thm}\label{theorem-010}
Let $f\in\mathcal{R}$ be of the form (\ref{p-001}). Then
\begin{enumerate}[(i)]
\item $\displaystyle |T_2(n)|\le \frac{4}{n^2}+\frac{4}{(n+1)^2}\quad$ for $n\ge2$.\\[1mm]

\item $\displaystyle |T_3(1)|\le \frac{35}{9}$.\\[1mm]

\item $\displaystyle |T_3(2)|\le \frac{7}{3}$.
\end{enumerate}
The inequalities in $(i)$ and $(ii)$ are sharp.
\end{thm}

\begin{proof}
Let $f\in\mathcal{R}$ be of the form (\ref{p-001}). Then there exists a function $p\in\mathcal{P}$ of the form (\ref{p-030}) such that $f'(z)=p(z)$. Equating  coefficients we obtain $na_n=c_{n-1}$, and so
$$
|a_n|=\frac{1}{n}|c_{n-1}|\le \frac{2}{n}, \quad n\ge 2.
$$
The inequality is sharp for the function $f$ defined by $f'(z)=(1+z)/(1-z)$, or its rotations. Thus
\begin{equation}\label{p-090}
|T_2(n)|=|a_n^2-a_{n+1}^2|\le |a_n^2|+|a_{n+1}^2|\le \frac{4}{n^2}+\frac{4}{(n+1)^2}.
\end{equation}
 Equality holds in (\ref{p-090}) for the function $f$ defined by
\begin{equation}\label{p-100}
f'(z):=\frac{1+iz}{1-iz}.
\end{equation}

Next, if $f\in\mathcal{R}$ is of the form (\ref{p-001}) then
\begin{align}\label{p-110}
|T_3(1)| &=|1-2a_2^2+2a_2^2a_3-a_3^2|\\
&\le 1+2|a_2^2|+|a_3||a_3-2a_2^2|\nonumber\\
&\le 1+2+\frac{2}{3}\left|\frac{1}{3}c_2-\frac{1}{2}c_1^2\right|\nonumber\\
&\le 3+\frac{2}{9}\left|c_2-\frac{3}{2}c_1^2\right|\nonumber\\
&\le 3+\frac{8}{9}=\frac{35}{9}.\nonumber
\end{align}
It is easy to see that  equality in (\ref{p-110}) holds for the function $f$ defined by (\ref{p-100}).
\bigskip

Again, if $f\in\mathcal{R}$ is of the form (\ref{p-001}) then
\begin{align*}
|T_3(2)| &=|a_2^3-2a_2a_3^2-a_2a_4^2+2a_3^2a_4|\\
&\le |a_2|^3 + 2|a_2||a_3^2| +|a_4||a_2a_4-2a_3^2|\\
&\le 1+\frac{8}{9}+\frac{1}{2}|a_2a_4-2a_3^2|\\
&\le \frac{17}{9}+\frac{1}{2}|a_2a_4-2a_3^2|.
\end{align*}

Thus we need to find the maximum value of $|a_2a_4-2a_3^2|$ for functions in $\mathcal{R}$. By (\ref{p-165}), it easily follows that
$$
|a_2a_4-2a_3^2|=\frac{1}{72}|9c_1c_3-16c_2^2|\le \frac{64}{72}=\frac{8}{9}.
$$
Therefore
$$
|T_3(2)|\le \frac{17}{9}+\frac{4}{9}=\frac{7}{3}.
$$
\end{proof}

\begin{rem}
The above theorem shows that for $f\in \mathcal{R}$, the sharp inequalities $|T_2(2)|\le 13/9$ and $|T_2(3)|\le 17/36$ hold. In \cite{Radhika-Sivasubramanian-Murugusundaramoorthy-Jahangiri-2016}, it was claimed that $|T_2(2)|\le 5/9$, $|T_2(3)|\le 4/9$, $|T_3(1)|\le 13/9$ and $|T_3(2)|\le 4/9$ hold for functions in $\mathcal{R}$ and these estimates  are sharp. For the function $f$ defined by (\ref{p-100}), a simple computation gives $|T_2(2)|= 13/9$, $|T_2(3)|= 17/36$, $|T_3(1)|= 35/9$ and $|T_3(2)|\le 25/12$,  showing that the these estimates  are not correct. As explained above,  the authors assumed that $c_1>0$, which is not justified, since the functional $|T_q(n)|$ $(n\ge1, q\ge2)$ is not rotationally invariant.
\end{rem}

\bigskip

If $f\in\mathcal{T}$ is given by (\ref{p-001}), then the coefficients of $f$ can be expressed  by
\begin{equation*}\label{p-210}
a_n=\int_{-1}^{1} \frac{\sin(n\arccos t)}{\sin(\arccos t)}\,d\mu(t) =\int_{-1}^{1} U_{n-1}(t)\,d\mu(t), \quad n\ge 1
\end{equation*}
where $U_{n}(t)$ are Chebyshev polynomials of degree $n$ of the second kind.
\bigskip

Let $A_{n,m}$ denote the region of variability of the point $(a_n,a_m)$, where $a_n$ and $a_m$ are  coefficients of a given function $f\in\mathcal{T}$ with the series expansion (\ref{p-001}), i.e., $A_{n,m}:=\{(a_n(f),a_m(f)):f\in\mathcal{T}\}$. Therefore, $A_{n,m}$ is the closed convex hull of the curve
$$
\gamma_{n,m}:[-1,1]\ni t\rightarrow (U_{n-1}(t),U_{m-1}(t)).
$$
\bigskip

By the Caratheodory theorem we conclude that it is sufficient to discuss only  functions
\begin{equation}\label{p-215}
F(z,\alpha,t_1,t_2):=\alpha k(z,t_1)+(1-\alpha)k(z,t_2),
\end{equation}
where $0\le\alpha\le1$ and $-1\le t_1\le t_2\le1$.

%
\bigskip

Let $X$ be a compact Hausdorff space, and $J_{\mu}=\int_{X} J(t)\,d\mu(t)$. Szapiel \cite{Szapiel-1986} proved the following theorem.

\begin{thm}\label{theorem-030}
Let $J:[\alpha,\beta]\rightarrow \mathbb{R}^n$ be continuous. Suppose that there exists a positive integer $k$, such that for each non-zero $\overrightarrow{p}$ in $\mathbb{R}^n$ the number of solutions of any equation $\langle \overrightarrow{J(t)},\overrightarrow{p}\rangle=const$, $\alpha\le t\le \beta$ is not greater than $k$. Then, for every $\mu\in P_{[\alpha,\beta]}$ such that $J_{\mu}$ belongs to the boundary of the convex hull of $J([\alpha,\beta])$, the following statements are true:
\bigskip

\begin{enumerate}
\item if $k=2m$, then
    \begin{enumerate}[(a)]
    \item $|supp (\mu)|\le m$, or
    \item $|supp (\mu)|=m+1$ and $\{\alpha,\beta\}\subset supp (\mu).$\\
    \end{enumerate}

\item if $k=2m+1$, then
    \begin{enumerate}[(a)]
    \item $|supp (\mu)|\le m$, or
    \item $|supp (\mu)|=m+1$ and one of the points $\alpha$ and $\beta$ belongs to $supp (\mu).$
    \end{enumerate}

\end{enumerate}

\end{thm}
\bigskip

In the above, the symbol $\langle \overrightarrow{u},\overrightarrow{v}\rangle$ means the scalar product of vectors $\overrightarrow{u}$ and $\overrightarrow{v}$, whereas the symbols $P_X$ and $|supp(\mu)|$ describe the set of probability measures on $X$, and the cardinality of the support of $\mu$, respectively.
\bigskip

Putting $J(t)=(U_{1}(t),U_{2}(t))$, $t\in[-1,1]$ and $\overrightarrow{p}=(p_1,p_2)$, we can see that any equation  of the form $p_1U_{1}(t)+p_2U_{2}(t)=const$, $t\in[-1,1]$ has at most $2$ solutions. According to Theorem \ref{theorem-030}, the boundary of the convex hull of $J([-1,1])$ is determined by atomic measures $\mu$ for which support consists of at most 2 points. Thus we have the following.

\begin{lem}\label{p-lemma055}
The boundary of $A_{2,3}$ consists of points $(a_2,a_3)$ that correspond to the functions $F(z,1,t,0)=k(z,t)$ or $F(z,\alpha,1,-1)$ with $0\le\alpha\le1$ and $-1\le t\le1$ where $F(z,\alpha,t_1,t_2)$ is defined by (\ref{p-215}).
\end{lem}

In a similar way, one can obtain the following:

\begin{lem}\label{p-lemma060}
The boundary of $A_{3,4}$ consists of points $(a_3,a_4)$ that correspond to the functions $F(z,\alpha,t,-1)$ or $F(z,\alpha,t,1)$ with $0\le\alpha\le1$ and $-1\le t\le1$ where $F(z,\alpha,t_1,t_2)$ is defined by (\ref{p-215}).
\end{lem}

Before we proceed further, we give some example of typically real functions.

\begin{example}\label{example-1}
For each $t\in[-1,1]$, the function $k(z,t)=z/(1-2tz+z^2)$ is a typically real function. For the function $k(z,1)=z/(1-z)^2$, we have $T_2(n)=n^2-(n+1)^2=-(2n+1)$, and $T_3(n)=a_n^3-2 a_{n+1}^2 a_n-a_{n+2}^2 a_n+2 a_{n+1}^2 a_{n+2}=4(n+1)$.
\end{example}

\begin{example}\label{example-2}
The function $f(z)=-\log(1-z)=z+\sum_{n=2}^{\infty} (1/n)z^n$ is a typically real function. For this function, we have $T_2(n)=1/n^2-1/(n+1)^2$ and $T_3(n)=4 \left(n^2+3 n+1\right)/(n^3 (n+1)^2 (n+2)^2)$.
\end{example}

\bigskip

\begin{lem}\label{p-lemma065}
If $f\in\mathcal{T}$ then $T_2(n)$ attains its extreme values on the boundary of $A_{n,n+1}$.
\end{lem}

\begin{proof}
Let $\phi(x,y)=x^2-y^2$, where $x=a_n$ and $y=a_{n+1}$. The only critical point of $\phi$ is $(0,0)$ and $\phi(0,0)=0$. Since $\phi$ may be positive  as well as negative for $(x,y)\in A_{n,n+1}$ (see Example \ref{example-1} and Example \ref{example-2}), the extreme values of $\phi$ are attained on the boundary of $A_{n,n+1}$.
\end{proof}

In a similar way, we can prove the following:

\begin{lem}\label{p-lemma065}
If $f\in\mathcal{T}$ then $T_3(1)$ attains its extreme values on the boundary of $A_{2,3}$.
\end{lem}

Since all coefficients of $f\in\mathcal{T}$ are real, we look for the lower and the upper bounds of $T_q(n)$ instead of the bound of $|T_q(n)|$. The proof of the following theorem is obvious.

\begin{thm}\label{theorem-040}
For every function $f\in\mathcal{T}$ of the form (\ref{p-001}), we have $-(n+1)^2\le T_2(n)\le n^2$.
\bigskip

In particular
\begin{enumerate}[(i)]
\item if $n$ is odd then $\max \{T_2(n):f\in\mathcal{T}\}=n^2$ and equality attained for the function $F(z,1/2,1,-1)$.

\item if $n$ is even then $\min \{T_2(n):f\in\mathcal{T}\}=-(n+1)^2$ and equality attained for the function $F(z,1/2,1,-1)$.
\end{enumerate}
\end{thm}

\begin{thm}\label{theorem-045}
For $f\in \mathcal{T}$,  $\ds \max \{T_2(2):f\in\mathcal{T}\}=\frac{5}{4}$.
\end{thm}

\begin{proof}
By Lemma \ref{p-lemma055}, it is enough to consider the functions $F(z,1,t,0)=k(z,t)$ and $F(z,\alpha,1,-1)$ with $0\le\alpha\le1$ and $-1\le t\le1$.
\bigskip

\textbf{Case 1.} For the function $F(z,1,t,0)=k(z,t)=z+2 t z^2+\left(4 t^2-1\right) z^3+\left(8 t^3-4 t\right) z^4+\cdots$, we have
$$
a_2^2-a_3^2=-16 t^4+12 t^2-1\le \frac{5}{4}.
$$

\textbf{Case 2.} For the function $F(z,\alpha,1,-1)=z+ (4 \alpha -2) z^2+3 z^3+(8 \alpha -4) z^4+\cdots$, we have
$$
a_2^2-a_3^2=(2-4 \alpha )^2-9\le -5.
$$

\noindent The conclusion follows from Cases 1 and  2, with the maximum attained for the function $F(z,1,t,0)=k(z,t)$ with $\ds t=\frac{\sqrt3}{2\sqrt2}$.
\end{proof}

\begin{cor}
For  $f\in \mathcal{T}$, we have the sharp inequality $-9\le T_2(2)\le \frac{5}{4}$.
\end{cor}

\bigskip

\begin{thm}\label{theorem-050}
For  $f\in \mathcal{T}$, we have $\ds \min \{T_2(3):f\in\mathcal{T}\}=-7$.
\end{thm}

\begin{proof}
By Lemma \ref{p-lemma060}, it is enough to consider the functions $F(z,\alpha,t,-1)$ and $F(z,\alpha,t,-1)$ with $0\le\alpha\le1$ and $-1\le t\le1$.
\bigskip

\textbf{Case 1.} For the function $F(z,\alpha,t,-1)=z+2(\alpha +\alpha  t-1)z^2+ \left(4\alpha t^2-4\alpha+3\right)z^3+ \left(4\alpha+8\alpha t^3-4 \alpha t-4\right)z^4+\cdots$, we have
$$
T_2(3)=a_3^2-a_4^2=\left(4\alpha t^2-4\alpha+3\right)^2-\left(4\alpha+8\alpha t^3-4 \alpha t-4\right)^2:=\phi(\alpha,t).
$$
By elementary calculus, one can verify that
$$
\min_{0\le\alpha\le1,-1\le t\le1} \phi(\alpha,t)=\phi(0,0)=-7.
$$

\textbf{Case 2.} For the function $F(z,\alpha,t,1)=z+ 2(1-\alpha +\alpha  t) z^2+\left(3-4\alpha+4\alpha t^2\right)z^3+\left(4-4\alpha-4 \alpha t+8\alpha t^3\right) z^4+\cdots$, we have $a_2^2-a_3^2=\phi(\alpha,-t)$, and so
$$
\min_{0\le\alpha\le1,-1\le t\le1} \phi(\alpha,-t)=\phi(0,0)=-7.
$$
The conclusion follows from Cases 1 and  2, and the maximum is attained for the function $F(z,0,0,1)$ or $F(z,0,0,-1)$.
\end{proof}

\begin{cor}
For  $f\in \mathcal{T}$, we have the sharp inequality $-7\le T_2(3)\le 9$.
\end{cor}

\bigskip

\begin{thm}\label{theorem-055}
For  $f\in \mathcal{T}$, we have $\ds \max \{T_3(1):f\in\mathcal{T}\}=8$, and $\ds \min \{T_3(1):f\in\mathcal{T}\}=-8$.
\end{thm}

\begin{proof}
By Lemma \ref{p-lemma055}, it is enough to consider the functions $F(z,1,t,0)=k(z,t)$ and $F(z,\alpha,1,-1)$ with $0\le\alpha\le1$ and $-1\le t\le1$.
\bigskip

\textbf{Case 1.} For the function $F(z,1,t,0)=k(z,t)=z+2 t z^2+\left(4 t^2-1\right) z^3+\left(8 t^3-4 t\right) z^4+\cdots$, we have $T_3(1)=1-2a_2^2+2a_2^2a_3-a_3^2=8 t^2 \left(2 t^2-1\right):=\phi_1(t)$,
and it is easy to verify that
$$
\max_{-1\le t\le1} \phi_1(t)=\phi_1(-1)=8\quad\mbox{and}\quad \min_{-1\le t\le1} \phi_1(t)=\phi_1(-1/2)=-1.
$$

\textbf{Case 2.} For the function $F(z,\alpha,1,-1)=z+ (4 \alpha -2) z^2+3 z^3+(8 \alpha -4) z^4+\cdots$, we have $T_3(1)=8 \left(8 \alpha ^2-8 \alpha +1\right):=\psi_1(\alpha)$,
\bigskip

\noindent and it is again easy to verify that
$$
\max_{0\le\alpha\le1} \psi_1(\alpha)=\psi_1(0)=8\quad\mbox{and}\quad \min_{0\le\alpha\le1} \psi_1(\alpha)=\psi_1(1/2)=-8.
$$
The conclusion follows from Cases 1 and  2, and the maximum is attained for the function $F(z,1,-1,0)=k(z,-1)$, and the minimum is attained for the function $F(z,1/2,1,-1)$.
\end{proof}

\noindent\textbf{Acknowledgement:}
The authors thank Prof. K.-J. Wirths for useful discussion and suggestions.

\end{document}